\documentclass[12pt]{article}
\usepackage[latin1]{inputenc}
\usepackage[T1]{fontenc}
\usepackage{amsmath,amssymb}
\usepackage{amsfonts,amsthm}

\usepackage{geometry}
\usepackage{color}
 \definecolor{cites}{rgb}{0.75 , 0.00 , 0.00}  
 \definecolor{urls} {rgb}{0.00 , 0.00 , 1.00}  
 \definecolor{links}{rgb}{0.00 , 0.00 , 0.5}   
 \definecolor{gray}{rgb}{0.5,0.5,.5}
\usepackage[plainpages=false,colorlinks=true, citecolor=cites, urlcolor=urls, linkcolor=links, breaklinks=true]{hyperref}

\newcommand{\C}{\mathbb{C}}
\newcommand{\N}{\mathbb{N}}
\newcommand{\T}{\mathbb{T}}
\newcommand{\Z}{\mathbb{Z}}
\newcommand{\Dc}{\mathcal{D}}

\renewcommand{\epsilon}{\varepsilon}

\DeclareMathOperator{\dist}{dist}
\DeclareMathOperator{\spann}{span}
\DeclareMathOperator*{\esssup}{ess-sup}

\newcommand{\from}{\colon}

\providecommand{\scpr}[2]{\left\langle #1, #2 \right\rangle}
\renewcommand{\sp}{\scpr}
\providecommand{\abs}[1]{\left\lvert#1\right\rvert}
\providecommand{\norm}[1]{\left\lVert#1\right\rVert}
\providecommand{\set}[1]{\left\{ #1\right\}}

\newtheorem{thm}{Theorem}

\newtheorem{lem}[thm]{Lemma}
\newtheorem{prop}[thm]{Proposition}
\newtheorem{cor}[thm]{Corollary}

\theoremstyle{definition}

\theoremstyle{remark}
\newtheorem{rem}[thm]{Remark}

\allowdisplaybreaks

\title{Bounded and compact Toeplitz+Hankel matrices}

\author{T. Ehrhardt\footnote{Mathematics Department, University of California, Santa Cruz, California, CA-95064, USA, \texttt{tehrhard@ucsc.edu}}\ , R. Hagger\footnote{Department of Mathematics, University of Reading, Reading RG6 6AX, United Kingdom, \texttt{r.t.hagger@reading.ac.uk}}\ ~and J.A. Virtanen\footnote{Department of Mathematics, University of Reading, Reading RG6 6AX, United Kingdom, \texttt{j.a.virtanen@reading.ac.uk} and Department of Mathematics, University of Helsinki, Helsinki 00014, Finland, \texttt{jani.virtanen@helsinki.fi}}}

\date{\today}

\begin{document}

\maketitle

\makeatletter
\def\blfootnote{\gdef\@thefnmark{}\@footnotetext}
\makeatother

\abstract{We show that an infinite Toeplitz+Hankel matrix $T(\varphi) + H(\psi)$ generates a bounded (compact) operator on $\ell^p(\N_0)$ with $1\leq p\leq \infty$ if and only if both $T(\varphi)$ and $H(\psi)$ are bounded (compact). We also give analogous characterizations for Toeplitz+Hankel operators acting on the reflexive Hardy spaces. In both cases, we provide an intrinsic characterization of bounded operators of Toeplitz+Hankel form similar to the Brown-Halmos theorem. In addition, we establish estimates for the norm and the essential norm of such operators.

\medskip


\medskip


\medskip

\noindent\textit{Acknowledgements:} This project has received funding from the European Union's Horizon 2020 research and innovation programme under the Marie Sklodowska-Curie grant agreement No 844451. Virtanen was supported in part by EPSRC grant EP/T008636/1. Ehrhardt was supported in part by the Simons Foundation Collaboration Grant \# 525111. The authors also thank the American Institute of Mathematics and the SQuaRE program for their support.

\medskip
}

\section{Introduction}

In 1911, in a footnote of~\cite{Toeplitz}, Otto Toeplitz proved that the infinite matrix of the form $(\varphi_{j-k})_{j,k \in \N_0}$ generates a bounded operator\footnote{More precisely, Toeplitz showed that the finite sections of the matrix $(\varphi_{j-k})_{j,k\geq 0}$ are uniformly bounded if and only if the same is true for the Laurent matrix $(\varphi_{j-k})_{j,k\in\Z}$, which by the Banach-Steinhaus theorem characterizes boundedness.} on $\ell^2(\N_0)$ if and only if there is a function $a\in L^\infty$ of the unit circle $\T$ such that $\hat{a}_j = \varphi_j$ for all $j\in\Z$, where $\hat{a}_j$ is the $j$-th Fourier coefficient of the function $a$. No one could have predicted the enormous impact that these and related seemingly simple matrices have had in mathematics and its applications for the past five decades, see e.g.~\cite{BoeSi2006, DIK2011, DIK2013, Widom} and the references therein. Indeed, about 50 years after Toeplitz' result, the same characterization was given independently by Brown and Halmos~\cite{BH, Halmos}, who viewed Toeplitz matrices $T(a) = (\hat{a}_{j-k})_{j,k\geq 0}$ as operators acting on the Hardy space $H^2$ of the unit circle $\T$ by the rule
$$
	T(a)f = PM(a)f = P(af),
$$
where $P$ is the orthogonal (Riesz) projection of $L^2$ onto $H^2$ and $M(a)$ is the operator of multiplication by $a$. Notice that for $j,k \geq 0$, we have
$$
	\langle T(a) \chi_k, \chi_j \rangle = \langle a, \chi_{j-k}\rangle = {\hat{a}}_{j-k},
$$
where $\{\chi_n : n\in\N_0\}$ is the standard orthonormal basis of $H^2$. More precisely, Brown and Halmos showed that if $A$ is a bounded linear operator on $H^2$ and if there are $\varphi_n\in\C$ such that $\langle A\chi_k, \chi_j \rangle = \varphi_{j-k}$ for all $j,k\geq 0$, then there is an $a\in L^\infty$ such that $A = T(a)$ and $\hat{a}_n = \varphi_n$ for all $n\in\Z$. Moreover, $\|T(a)\| = \|a\|_\infty$. 

Since the result of Brown and Halmos, the study of Toeplitz operators has been extended to the other Hardy spaces
$$
	H^p = \{ f\in L^p : \hat{f}_n = 0\ {\rm for}\ n<0\},
$$
sequence spaces $\ell^p(\N_0)$, and various other (analytic) function spaces. Indeed, for $1<p<\infty$, Brown and Halmos' result for $T(a)$ acting on the reflexive Hardy spaces $H^p$ takes the same form as in the case $p=2$ with the norm estimate
\begin{equation}\label{e:Toeplitz norm}
	\|a\|_\infty \leq \|T(a)\| \leq \|P\|_{B(L^p)} \|a\|_\infty,
\end{equation}
where
\begin{equation}\label{e:c_p}
	c_p := \|P\|_{B(L^p)} = \frac1{\sin \frac\pi{p}}
\end{equation} 
(see~\cite[Theorem~2.7]{BoeSi2006} for the norm estimate and~\cite{HoVe} for the exact value of $c_p$). In particular, $T(a)$ is bounded on $H^p$ if and only if $a\in L^\infty$. For the boundedness of Toeplitz operators on the other Hardy spaces $H^p$ with $0<p\leq 1$, see~\cite{JPS, PV, Tol}.

Regarding the essential norm, which is defined as
$$
	\norm{A}_{\mathrm{ess}}:=\inf \left\{ \norm{A+K}\,:\, K \mbox{ compact operator}\right\},
$$
it is easy to see that for $a\in L^\infty$
\begin{equation}\label{e:ess norm of T}
	\|a\|_\infty \leq \|T(a)\|_{\mathrm{ess}} \leq c_p\|a\|_\infty,
\end{equation}
where the first inequality follows from~\cite[Theorem~2.30]{BoeSi2006} while the other one follows from~\eqref{e:Toeplitz norm}. As was communicated to us by E.~Shargorodsky, the estimate \eqref{e:ess norm of T} is sharp. This can be seen by considering the function $a(e^{it}) = \sin\frac{\pi}p \pm i \cos \frac{\pi}p$, $\pm t \in (0,\pi)$ and using the standard Fredholm theory of piecewise continuous symbols. For further results on the essential norm of Toeplitz operators, see~\cite{Sha}.

To define Toeplitz operators $T(a)$ on $\ell^p(\N_0)$ with $1\leq p\leq\infty$, we denote by $M^p$ the set of all $a\in L^1$ for which $L(a)$, defined by $L(a)\psi := a\ast\psi$ is a bounded operator on $\ell^p(\Z)$. The discrete convolution $\ast$ is defined by
$$
	(a\ast \psi)_j := \sum_{k\in\Z} \hat a_{j-k} \psi_k\quad (j\in\Z).
$$
 For $a\in M^p$, we define $T(a) : \ell^p(\N_0)\to \ell^p(\N_0)$ by
$$
	T(a) \psi = P(a\ast \psi) = \left( \sum_{k=0}^\infty \hat{a}_{j-k} \psi_k \right)_{j\in\N_0}
$$
where $P(\ldots, x_{-1}, x_0, x_1, \ldots) = (x_0, x_1, \ldots)$ is the standard projection of $\ell^p(\Z)$ onto $\ell^p(\N_0)$. In 1975 Duduchava~\cite{Duduchava} proved the following analog of the Brown--Halmos theorem when $1\leq p<\infty$: If $A$ is bounded on $\ell^p(\N_0)$ and if there is a sequence of complex numbers $\varphi_n$ such that $\langle Ae_k, e_j\rangle = \varphi_{j-k}$ for all $j,k\in\N_0$, then there is an $a\in M^p$ such that $A = T(a)$, $\hat{a}_n = \varphi_n$ for all $n\in\Z$, and
$$
	\|T(a)\| = \|L(a)\|.
$$
Here, $\{e_j : j \in \Z\}$ denotes the standard orthonormal basis of $\ell^2(\Z)$ and $\langle \cdot, \cdot\rangle$ is the dual pairing induced by the inner product on $\ell^2(\Z)$. For the proof of Duduchava's result see again \cite[Theorem~2.7]{BoeSi2006}. To obtain the essential norm of $T(a)$ one may proceed as in \cite[Section 3.7.2]{Lindner2006}, which yields $\|T(a)\|_{\mathrm{ess}} = \|L(a)\|$ for $a \in M^p$.

The result of Duduchava can in principle be transferred to $p = \infty$ as well. However, there is one caveat. A bounded linear operator on $\ell^{\infty}(\N_0)$ is not uniquely determined by the matrix elements $(\langle Ae_k, e_j\rangle)_{j,k \in \N_0}$. In other words, not every bounded linear operator on $\ell^{\infty}(\N_0)$ is given by an infinite matrix. Indeed, there are nonzero bounded linear operators $A$ on $\ell^{\infty}(\N_0)$ such that $\langle Ae_k, e_j\rangle = 0$ for all $j,k \in \N_0$. We refer to Chapter 1 of \cite{Lindner2006} for a discussion. As a consequence, Duduchava's result needs an additional assumption in case $p = \infty$. Namely, we need to assume that $A$ is given by an infinite matrix, that is,
\[(Ax)_j = \sum\limits_{k \in \N_0} \langle Ae_k, e_j\rangle x_k.\]
For such operators Duduchava's result is easily verified considering that the $\ell^{\infty}$-norm is 
the maximum absolute row sum norm in that case. Hence $M^\infty$ coincides with the Wiener algebra $W$ (see~\eqref{e:Mp inclusions} below) and in that case
\[\|T(a)\|_{\mathrm{ess}} = \|T(a)\| = \|L(a)\| = \sum\limits_{j \in \Z} |\hat{a}_j|\]
for $a \in W$.

Finally, we note that Toeplitz operators $T(a)$ are never compact unless $a=0$, which was proved in~\cite{BH} for $p=2$ and can be proved similarly for the other values of $p$. This well-known result also follows as a special case of Theorem \ref{thm:Toeplitz+Hankel_norm} (Theorem \ref{thm:Toeplitz+Hankel_compact_Hardy_1}, respectively) below.

Toeplitz operators are closely related to Hankel operators $H(a)$, which can be defined on Hardy spaces by setting
\begin{equation}\label{e:Hankel operator}
	H(a)f := PM(a)Jf = P(aJf),
\end{equation}
where $J$ is the flip operator, which is defined by $Jf(t) = \bar t f(\bar t)$ for $t\in \T$. It is easy to see that for $j,k\geq 0$,
$$
	\langle H(a) \chi_k, \chi_j \rangle = \langle a, \chi_{j+k+1} \rangle = \hat{a}_{j+k+1}.
$$
The Hankel operator $H(a)$ on $\ell^p(\N_0)$ is defined by
$$
	H(a) \varphi = \left( \sum_{k\in\N_0} \hat{a}_{j+k+1}\varphi_k\right)_{j\geq 0}.
$$
In 1957, Nehari~\cite{Nehari} showed that the infinite matrix $(a_{j+k+1})_{j,k\geq 0}$ generates a bounded operator on $\ell^2(\N_0)$ if and only if there is a function $b\in L^\infty$ such that $b_n = a_n$ for $n\in\N$. For $1<p<\infty$, Nehari's characterization extends to Hankel operators acting on $H^p$. However, there is no Nehari-type result for Hankel operators acting on $\ell^p(\N_0)$ for $p \notin \{1,2,\infty\}$ (see Open Problem~2.12 of~\cite{BoeSi2006}). For $p \in \{1,\infty\}$ one can easily see that $H(a)$ is bounded if and only if $a$ is in the Wiener algebra.

Unlike for Toeplitz operators, there is a large class of compact Hankel operators, and, more precisely, as shown by Hartman~\cite{Hartman} in 1958 for $p=2$, the Hankel operator $H(a)$ is compact on $H^p$ with $1<p<\infty$ if and only if there is a continuous function $b$ such that $\hat{a}_n = \hat{b}_n$ for all $n\in\N$. However, as for boundedness, there is no characterization of compact Hankel operators on $\ell^p(\N_0)$ in terms of their symbols for $p \notin \{1,2,\infty\}$. For further details on Hankel operators, see Section~2.3 of~\cite{BoeSi2006} and also~\cite{Peller}.

If $T(a)$ and $H(b)$ are both bounded, then obviously the \emph{Toeplitz+Hankel operator} $T(a) + H(b)$ is bounded. It is natural to ask whether the converse statement is true. We answer this question in the affirmative and also show that the compactness of $T(a)+H(b)$ implies that both $T(a)$ and $H(b)$ are compact (and then, necessarily, $a=0$). 
Furthermore, we obtain  estimates for the norm and the essential norm of bounded Toeplitz+Hankel operators. 
For the spectral properties of $T(a) + H(b)$ with bounded symbols, see~\cite{BE2013, DS, DS2}  and the references therein. In addition to intrinsic interest, these types of operators play a role in the asymptotic study of Toeplitz+Hankel determinants; see~\cite{BE2017, GI} and the references therein.

Our main result is an analog of the Brown--Halmos theorem. It gives an intrinsic characterization of bounded operators of Toeplitz+Hankel form with unique symbols (see Theorem~\ref{main1} and Theorem~\ref{thm:main2}).

We start with the sequence spaces $\ell^p(\N_0)$ by exploiting the structures of and connections between infinite Toeplitz, Laurent, Hankel, and checkerboard matrices, using operator theoretic methods, such as limit operator techniques. To deal with $p = \infty$ we use pre-adjoints which is essentially motivated by similar arguments in \cite{Lindner2006}.  For the treatment of Toeplitz+Hankel operators on Hardy spaces, we consider sesquilinear (Toeplitz, Hankel etc.) forms and obtain similar characterizations. However, we only deal with the case $1<p<\infty$ in Hardy space.

\section{Sequence spaces}

Throughout this section, we assume that $1\leq p\leq\infty$ unless stated otherwise. For an infinite set $J\subseteq\Z$, we say that the infinite matrix $(A_{j,k})_{j,k\in J}$ generates a bounded linear operator on $\ell^p(J)$ if there is a positive constant $c$ such that for each $\varphi\in \ell^p(J)$, the sequence $\psi=(\psi_j)_{j\in J}$, where $\psi_j= \sum_{k\in J} A_{j,k} \varphi_k$, is in $\ell^p(J)$ and $\|\psi\| \leq c \|\varphi\|$. In that case, there is a bounded linear operator $A$ acting on $\ell^p(J)$ such that
$$
	A_{j,k} = \langle Ae_k, e_j \rangle
$$
for all $j,k\in J$, where $(e_k)_{k\in J}$ is the standard basis for $\ell^2(J)$ and
$$
	\langle \varphi, \psi \rangle = \sum_{j\in J} \varphi_j \overline{\psi_j}
$$
for $\varphi\in \ell^p(J)$ and $\psi\in \ell^q(J)$ with $1/p+1/q=1$. If $p < \infty$, every bounded linear operator is generated by an infinite matrix. For $p = \infty$ this is not the case, but we will only consider operators that are generated by infinite matrices. We will say that an infinite matrix is bounded if the corresponding operator is bounded. The norm of the matrix is given by the operator norm.

For the sake of completeness, we will also briefly mention bounded linear operators $A:c_0(J)\to\ell^\infty(J)$  (see Remark \ref{rem:3}(b) below). Here $c_0(J)\subset \ell^\infty(J)$
stands for the set of sequences $\varphi=(\varphi_j)_{j\in J}$ for which $\lim\limits_{|j|\to+\infty} |\varphi_j|=0$. In this case, every bounded linear operator is generated by an infinite matrix because sequences of finite support are dense in $c_0(J)$.

For a bounded sequence $\varphi = (\varphi_n)_{n \in \Z}$, let $L(\varphi)$ denote the (formal) Laurent matrix given by
\[
	L(\varphi)_{j,k} = \varphi_{j-k} \quad (j,k \in \Z).
\]
Similarly, we define the Toeplitz matrix $T(\varphi)$ by
\[
	T(\varphi)_{j,k} = \varphi_{j-k} \quad (j,k \in \N_0)
\]
and the Hankel matrix $H(\varphi)$ by
\[
	H(\varphi)_{j,k} = \varphi_{j+k+1} \quad (j,k \in \N_0).
\]
In case these matrices define bounded operators on $\ell^p(\Z)$ or $\ell^p(\N_0)$, respectively, we call them Laurent, Toeplitz or Hankel \emph{operators}. A two-sided infinite matrix $C$ is called a checkerboard matrix if $C_{j,k} = C_{j+1,k+1} = C_{j+1,k-1}$ for all $j,k \in \Z$ (and likewise for one-sided infinite matrices $C$). Here, two-sided refers to $J = \Z$ and one-sided refers to $J = \N_0$.

We say that a sequence of matrices $A_n$ converges elementwise to a matrix $A$ if $(A_n)_{j,k}$ converges to $A_{j,k}$ for every $j,k \in J$. Similarly, a sequence $(\varphi_n)_{n \in \N}$ in $\ell^p(J)$ converges elementwise to $\varphi \in \ell^p(J)$ if $(\varphi_n)_j \to (\varphi)_j$ for all $j \in J$.

\begin{lem} \label{lem:elementwise_convergence_bounded}
If $(A_n)_{n \in \N}$ is a uniformly bounded sequence of matrices $A_n$ that converges elementwise to the matrix $A$, then $A$ generates a bounded linear operator and
$$
	\|A\| \leq \sup_{n\in\N} \|A_n\|.
$$
\end{lem}

\begin{proof}
First assume $p < \infty$. Let $M := \sup\limits_{n \in \N} \norm{A_n}$ and $\varphi \in \ell^p(J)$ with finite support. As $\varphi$ has finite support, the sequence $(A_n \varphi)_{n \in \N}$ converges elementwise to $A\varphi$. By Fatou's lemma, $A\varphi \in \ell^p(J)$ and $\norm{A\varphi}_p \leq M\norm{\varphi}_p$. As the elements with finite support are dense in $\ell^p(J)$, $A$ extends to a bounded operator on $\ell^p(J)$ and the norm estimate follows.

Now assume that $p = \infty$ and consider the Hermitian transpose of the matrix $A_n$ defined by $(A_n^{\triangleleft})_{j,k} = \overline{(A_n)_{k,j}}$. Notice that $A_n^{\triangleleft}$ generates an operator on $\ell^1(J)$ and is called a pre-adjoint of $A_n$ because $(A_n^{\triangleleft})^* = A_n$. It follows $\|A_n^{\triangleleft}\| = \|A_n\|$ for all $n \in \N$. We can therefore apply the above for $p  = 1$ to the sequence $(A_n^{\triangleleft})_{n \in \N}$ to get the required result.
\end{proof}

We denote by $\hat a$ the sequence of the Fourier coefficients of $a\in L^1$. Note that necessarily $\hat{a}\in c_0(\Z)$. For $1\leq p\leq\infty$, the class $M^p$ of symbols $a\in L^1$ defined in the previous section can be described as the set of all $a\in L^1$ for which the Laurent matrix $L(\hat a)$ generates a bounded operator on $\ell^p(\Z)$. This operator is of course exactly $L(a)$ as defined in the introduction.

It is well known that for $1\leq p \leq r \leq 2$ and $1/p+1/q = 1$, we have $M^p = M^q$,
\begin{equation}\label{e:Mp inclusions}
	W = M^1 \subseteq M^p \subseteq M^r \subseteq M^2 = L^\infty
\end{equation}
and
$$
	\|a\|_\infty \leq \|L(a)\|_r \leq \|L(a)\|_p \leq \|a\|_W := \sum_{j\in\Z} |\hat{a}_j|,
$$
where $W$ is the Wiener algebra of continuous functions $a$ for which $(\hat{a}_j)_{j \in \Z} \in \ell^1(\Z)$. From these, it follows that $M^p\subseteq L^\infty$. Each $M^p$ is a Banach algebra with $\|a\|_{M^p} := \|L(a)\|$ for $a\in M^p$. These properties of the algebras $M^p$ and the following converse statement can be found in~\cite[Sect.~2.3--2.5]{BoeSi2006}, for example.

\begin{prop}\label{prop:Laurent.discrete}
Let $1\leq p<\infty$, $A$ be a bounded operator on $\ell^p(\Z)$ and $\langle Ae_k, e_j\rangle = \varphi_{j-k}$ for all $j,k\in\Z$, where $(\varphi_j)_{j\in\Z}$ is a sequence of complex numbers. Then there is an $a\in M^p$ such that $A = L(a)$ and $\hat{a}_j = \varphi_j$ for all $j\in\Z$.  
\end{prop}

\begin{rem}\label{rem:3}
(a) 
Assuming that $A$ is generated by an infinite matrix, one can prove the same result for $p = \infty$ by proceeding as in Lemma \ref{lem:elementwise_convergence_bounded} and using that $M^1 = M^{\infty}=W$. 
Alternatively, it follows by just considering the maximum absolute row sum norm.

(b)
The result remains true also for bounded linear operators $A:c_0(\Z)\to \ell^\infty(\Z)$.
In this case, we obtain $A=L(a)$ with $a\in W$, and, as a consequence, $A$ maps $c_0(\Z)$ into itself.
In fact, all the results in this section remain true for bounded linear operators $A:c_0(J)\to \ell^\infty (J)$ with $J=\Z$ or $J=\N_0$.
To see this notice that any such operator can be extended via its matrix representation to a bounded linear  operator on $\ell^\infty(J)$.
We leave the detailed verification to the reader.
\hfill $\Box$
\end{rem}

For the following we will need the flip operator $J \from \ell^p(\Z) \to \ell^p(\Z)$, which is given by 
$$
	(J\varphi)_k = \varphi_{-k-1}.
$$

\begin{prop} \label{prop:Laurent_bounded}
For $1\le p\le \infty$, let $\varphi, \psi\in \ell^\infty(\Z)$ and assume $A = L(\varphi) + L(\psi)J$ generates a  bounded linear operator on $\ell^p(\Z)$. 
Then both $L(\varphi)$ and $L(\psi)$ can be written as a sum of a bounded Laurent matrix and a checkerboard matrix $C$, i.e.,
$$
L(\varphi)=L(\hat{a})+C,\qquad L(\psi)=L(\hat{b})-CJ
$$
with $a,b\in M^p$, and $\max\{ \norm{a}_{M^p},\norm{b}_{M^p}\}\le \norm{A}\le \norm{a}_{M^p}+\norm{b}_{M^p}.$
\end{prop}

\begin{proof}
Consider the standard shift $V = L(e_1)$. Then
\[
	A - VAV = L(\varphi) + L(\psi)J - VL(\varphi)V - VL(\psi)JV = L(\varphi-V^2\varphi).
\]
Similarly, $A - V^jAV^j = L(\varphi-V^{2j}\varphi)$. Clearly, $\norm{L(\varphi-V^{2j}\varphi)} \leq 2\norm{A}$. For $n \in \N$, consider the decomposition
\[
	L(\varphi) = \frac{1}{2n+1} \sum\limits_{j = -n}^n L(\varphi-V^{2j}\varphi) + \frac{1}{2n+1} \sum\limits_{j = -n}^n L(V^{2j}\varphi).
\]
Let $C_n := \frac{1}{2n+1} \sum\limits_{j = -n}^n L(V^{2j}\varphi)$. As $\varphi \in \ell^{\infty}(\Z)$, the matrix elements of $C_n$ are uniformly bounded by $\|\varphi\|_{\ell^{\infty}(\Z)}$. We can therefore choose an increasing sequence $(n_k)_{k \in \N}$ such that $(C_{n_k})_{k \in \N}$ converges elementwise to some matrix $C$. As
\begin{align*}
\left|(C_{n_k})_{l,m} - (C_{n_k})_{l+1,m-1}\right| &= \frac{1}{2n_k+1} \left|\sum\limits_{j = -n_k}^{n_k} \varphi_{l-m-2j} - \sum\limits_{j = -n_k}^{n_k} \varphi_{l-m-2j+2}\right|\\
&= \frac{1}{2n_k+1} \left|\varphi_{l-m-2n_k} - \varphi_{l-m+2n_k+2}\right|\\
&\leq \frac{2}{2n_k+1}\|\varphi\|_{\ell^{\infty}(\Z)},
\end{align*}
we get $(C)_{l,m} = (C)_{l+1,m-1}$ for all $l,m \in \Z$. Moreover, as every $C_n$ is a Laurent matrix, $C$ must be a Laurent matrix as well, that is, $(C)_{l,m} = (C)_{l+1,m+1}$. This implies that $C$ is a checkerboard matrix. Furthermore, the sequence $\left(L(\varphi) - C_{n_k}\right)_{k \in \N}$ converges elementwise to $L(\varphi) - C$. As
\[
	\norm{L(\varphi) - C_{n_k}} = \norm{\frac{1}{2n_k+1} \sum\limits_{j = -n_k}^{n_k} L(\varphi-V^{2j}\varphi)} \leq 2\norm{A}
\]
for all $k \in \N$, the limit $L(\varphi) - C$ is again bounded by Lemma \ref{lem:elementwise_convergence_bounded}. Moreover, since both $L(\varphi)$ and $C$ are Laurent matrices, $L(\varphi) - C$ is again a Laurent matrix. We conclude that 
$L(\varphi) - C =L(\hat{a})$ with $a\in M^p$  by Proposition~\ref{prop:Laurent.discrete}.

For $L(\psi)$, we use $J^2 = I$ to obtain
\[L(\psi) +CJ  = AJ - L(\varphi)J+CJ = (A - L(\hat{a}))J .\]
Here, the left hand side is a Laurent matrix and the right hand side is bounded. Therefore, it is equal to $L(\hat{b})$ for some $b\in M^p$ 
again by Proposition~\ref{prop:Laurent.discrete}. 

To obtain the norm estimate, observe that $A=L(\hat{a})+L(\hat{b})J$. As $\hat{b}\in c_0(\Z)$, the sequence 
$$
V^{-n} A V^n = L(\hat{a}) + L(e_{-2n}\ast \hat{b})J
$$
converges elementwise to $L(\hat{a})$. By Lemma \ref{lem:elementwise_convergence_bounded} this implies that $\norm{a}_{M^p}=\norm{L(\hat{a})}\le \norm{A}$.
Now consider $AJ=L(\hat{a})J+L(\hat{b})$ instead of $A$ to get $\norm{b}_{M^p}\le\norm{A}$. The upper estimate is trivial.
\end{proof}

\begin{thm} \label{thm:Toeplitz+Hankel_bounded}
For $1\le p\le \infty$, let
 $\varphi \in \ell^{\infty}(\Z)$, $\psi \in \ell^{\infty}(\N)$ and assume that $A := T(\varphi) + H(\psi)$ generates a bounded operator on $\ell^p(\N_0)$. 
Then $\varphi = \varphi_1 + \varphi_2$, where $\varphi_1$ is the Fourier sequence of an $M^p$-function and $T(\varphi_2)$ is a checkerboard matrix. Moreover, $\psi = \psi_1 + \psi_2$, where $H(\psi_1)$ is a bounded Hankel matrix with $\psi_1\in c_0(\N)$ and $H(\psi_2) = -T(\varphi_2)$. Both decompositions are unique.
\end{thm}

\begin{proof}
Let $P \from \ell^p(\Z) \to \ell^p(\N_0)$ be the canonical projection and $P^* \from \ell^p(\N_0) \to \ell^p(\Z)$ the inclusion. For $n \in \N$, consider $V^{-n}P^*APV^n$ and choose an increasing sequence $(n_k)_{k \in \N}$ such that 
$$
	\left(V^{-n_k}P^*APV^{n_k}\right)_{k \in \N}
$$ 
converges elementwise.
Clearly, the limit is of the form $L(\varphi) + L(\nu)J$ with $\nu \in l^{\infty}(\Z)$. Since
\[
\norm{V^{-n_k}P^*APV^{n_k}} \leq \norm{A}
\]
for all $k \in \N$, $L(\varphi) + L(\nu)J$ is again a bounded operator (see Lemma \ref{lem:elementwise_convergence_bounded}) 
and we can apply Proposition \ref{prop:Laurent_bounded}. Therefore, $L(\varphi) = L(\varphi_1) + L(\varphi_2)$, where $L(\varphi_1)$ is bounded and $L(\varphi_2)$ is a checkerboard matrix. This implies that $T(\varphi_1)$ is bounded, $\varphi_1$ is the Fourier sequence of an $M^p$-function by Proposition \ref{prop:Laurent_bounded} and $T(\varphi_2)$ is a checkerboard matrix. Using $H(\psi) = (T(\varphi) + H(\psi)) - T(\varphi_1) - T(\varphi_2)$, we get the decomposition for $\psi$. The uniqueness follows from the obvious fact that the only bounded checkerboard matrix is $0$. Finally notice that if $H(\psi_1)$ is a bounded Hankel matrix,
then $\psi_1\in c_0(\N)$, which can be seen by considering the first row or column.
\end{proof}

\begin{rem} \label{rem:Hankel_bounded}
(a) As a consequence, $A=T(\hat{a})+H(\psi_1)$ with $a\in M^p$ and $\psi_1\in c_0(\N)$ where $H(\psi_1)$ is a bounded Hankel matrix.

(b) Note that the precise condition for the boundedness of Hankel operators on $\ell^p(\N_0)$, $p \notin \{1,2,\infty\}$ is unknown (cf.~\cite[Open problem 2.12]{BoeSi2006}). For $p = 2$ we can use that $H(\psi_1)$ is unitarily equivalent to a Hankel operator $H(a)$ on $H^2$ and therefore $\psi_1$ is the (positive) Fourier sequence of a function $a \in L^{\infty}$ \cite[Theorem 2.11]{BoeSi2006}. For $p \in \{1,\infty\}$ it is easily seen that $H(\psi_1)$ is bounded if and only if $\psi_1 \in \ell^1(\Z)$. In that case $\psi_1$ is the Fourier sequence of a function in the Wiener algebra and $H(\psi_1)$ is also compact.
\hfill $\Box$
\end{rem}

\begin{thm}\label{thm:Toeplitz+Hankel_norm}
For $1\leq p\leq \infty$, let $A=T(\hat{a})+H(\psi_1)$ where $a\in M^p$ and $\psi_1\in c_0(\N)$ such that $H(\psi_1)$ generates a bounded Hankel matrix on $\ell^p(\N_0)$. Then, for any compact operator $K$,
\[
\max\{ \norm{a}_{M^p},\frac{1}{2}\norm{H(\psi_1)+K} \} \le \norm{A+K}\le  \norm{a}_{M^p}+\norm{H(\psi_1)+K}.
\] 
\end{thm}
\begin{proof}
{}Proceeding as in the last part of the proof of Proposition \ref{prop:Laurent_bounded}, we observe that the elementwise limit of $V^{-n} P^*AP V^n$ is $L(\hat{a})$ since $\psi_1\in c_0(\N)$. Furthermore, the elementwise limit of $V^{-n}P^* KPV^n$ is zero. Indeed, this is clear for $1<p<\infty$ as the operators $V^{\pm n}\to 0$ weakly as $n\to\infty$. For $p=\infty$ notice that, for fixed $k$,  the sequence $V^ne_k\to 0$ converges weakly on the space $\ell^\infty(\Z)$, hence $\norm{KPV^ne_k}\to0$. Therefore $\langle V^{-n} P^* K PV^n e_k,e_j\rangle\to 0$ for each fixed $k,j\in\Z$. The case $p=1$ is reduced to the case $p=\infty$ by passing to the adjoint and noticing that $\langle V^{-n} P^* K PV^n e_k,e_j\rangle= \langle  e_k, V^{-n} P^* K^* PV^ne_j\rangle \to 0$
by the previous argument.

Therefore, $V^{-n} P^*(A+K)P V^n\to L(\hat{a})$ elementwise, and  from Lemma \ref{lem:elementwise_convergence_bounded} we obtain
$\norm{a}_{M^p}=\norm{L(\hat{a})}\le \norm{A+K}$.
To proceed consider
$$
\norm{H(\psi_1)+K}\le \norm{A+K}+\norm{T(\hat{a})}\le  \norm{A+K}+\norm{a}_{M^p}\le 2\norm{A+K},
$$ 
which settles the lower estimate. The upper estimate is trivial.
\end{proof}

\begin{cor}
Let $1\leq p\leq \infty$, $\varphi \in \ell^{\infty}(\Z)$, $\psi \in \ell^{\infty}(\N)$ and assume that $A := T(\varphi) + H(\psi)$ generates a compact operator on $\ell^p(\N_0)$. Then $T(\varphi)$ is a checkerboard matrix and $\psi = \psi_1 + \psi_2$, where $H(\psi_1)$ is a compact Hankel matrix with $\psi_1\in c_0(\N)$ and $H(\psi_2) = -T(\varphi)$. Moreover, the decomposition of $\psi$ is unique.
\end{cor}
\begin{proof}
We know from Theorem \ref{thm:Toeplitz+Hankel_bounded} that we can decompose $\varphi=\hat{a}+\varphi_2$ with $a\in M^p$
and $\psi=\psi_1+\psi_2$ such that $H(\psi_2)=-T(\varphi_2)$ is checkerboard and $H(\psi_1)$ is bounded. So, $A=T(\hat{a})+H(\psi_1)$. Now apply
Theorem \ref{thm:Toeplitz+Hankel_norm} with $K=-A$ to conclude that $a=0$ and $H(\psi_1)=-K$ is compact.
\end{proof}

\begin{rem}
As for the boundedness, we remark that the precise condition for the compactness of Hankel operators on $\ell^p(\N_0)$, $p \notin \{1,2,\infty\}$ is unknown (cf.~\cite[Open problem 2.56]{BoeSi2006}). For $p = 2$ the Hankel operator $H(\psi_1)$ is compact if and only if $\psi$ is the (positive) Fourier sequence of a function $a \in C(\T) + \overline{H^{\infty}}$; see~\cite[Theorem 2.54]{BoeSi2006}. For $p \in \{1,\infty\}$ see Remark \ref{rem:Hankel_bounded}.
\hfill $\Box$
\end{rem}

Finally, we add a characterization involving displacement relations that also nicely summarizes this section. Statement (i) provides an intrinsic characterization of bounded operators of Toeplitz+Hankel form. The underlying ``displacement transform'' shown in \eqref{TH.displacement} is well-known in the theory of finite Toeplitz+Hankel matrices, see~\cite{HR1, HR2}.

\begin{thm}\label{main1}
For $1\le p\le\infty$, the following statements are equivalent:
\begin{itemize}
\item[{\rm (i)}]
$A$ is a bounded linear operator on $\ell^p(\N_0)$ with matrix representation $A=(A_{j,k})_{j,k\in\N_0}$ satisfying
the displacement relations
\begin{equation}\label{TH.displacement}
A_{j-1,k}+A_{j+1,k}-A_{j,k-1}-A_{j,k+1}=0 \quad (j,k \in \N).
\end{equation}
\item[{\rm (ii)}]
$A$ is a bounded linear operator on $\ell^p(\N_0)$ whose matrix representation can be written as $T(\varphi)+H(\psi)$
for some $\varphi\in \ell^\infty(\Z)$ and $\psi\in \ell^\infty(\N)$.
\item[{\rm (iii)}]
$A=T(\hat{a})+H(\varrho)$ for some $a\in M^p$ and $\varrho\in c_0(\N)$ such that $H(\varrho)$ is bounded on $\ell^p(\N_0)$.
\end{itemize}
Furthermore, the function $a$ and the sequence $\varrho$ in {\rm (iii)} are uniquely determined,  and 
\begin{align*}
\norm{T(\hat{a})+H(\varrho)}  & \simeq
\max\{ \norm{a}_{M^p},\norm{H(\varrho)} \},
\\[1ex]
\norm{T(\hat{a})+H(\varrho)}_{\mathrm{ess}}  &\simeq
\max\{ \norm{a}_{M^p},\norm{H(\varrho)}_{\mathrm{ess}} \},
\end{align*}
where ``$\simeq$'' stands for equivalence of norms.
\end{thm}

\begin{proof}
The implications (iii)$\Rightarrow$(ii)$\Rightarrow$(i) are obvious.

(i)$\Rightarrow$(ii):
Choose
\[\psi_{j} = \begin{cases} A_{\frac{j-1}{2},\frac{j-1}{2}} & \text{if $j$ is odd}, \\ A_{\frac{j-2}{2},\frac{j}{2}} & \text{if $j$ is even}. \end{cases}\]
Let $T = A - H(\psi)$. It remains to show that $T$ is a Toeplitz matrix. Assume that $j+k$ is even. Iterating the displacement relations gives
\begin{equation} \label{eq:iterated_displacement}
A_{j+1,k+1} - A_{j,k} = A_{j,k+2} - A_{j-1,k+1} = \ldots = A_{j-n+1,k+n+1} - A_{j-n,k+n}
\end{equation}
for all $n \in \{-k,\ldots,j\}$. Using \eqref{eq:iterated_displacement} with $n = \frac{j-k}{2}$, we get
\begin{align*}
T_{j+1,k+1} - T_{j,k} &= A_{j+1,k+1} - A_{j,k} - \psi_{j+k+3} + \psi_{j+k+1}\\
&= A_{j+1,k+1} - A_{j,k} - A_{\frac{j+k}{2}+1,\frac{j+k}{2}+1} + A_{\frac{j+k}{2},\frac{j+k}{2}}\\
&= 0.
\end{align*}
Similarly, if $j+k$ is odd, then
\begin{align*}
T_{j+1,k+1} - T_{j,k} &= A_{j+1,k+1} - A_{j,k} - \psi_{j+k+3} + \psi_{j+k+1}\\
&= A_{j+1,k+1} - A_{j,k} - A_{\frac{j+k+1}{2},\frac{j+k+3}{2}} + A_{\frac{j+k-1}{2},\frac{j+k+1}{2}}\\
&= 0,
\end{align*}
using \eqref{eq:iterated_displacement} with $n = \frac{j-k+1}{2}$. Hence $T = T(\varphi)$ for some bi-infinite sequence $\varphi = (\varphi_j)_{j \in \Z}$. Note that $\psi$ and $\varphi$ are bounded by construction because $|A_{j,k}| \leq \|A\|$ for all $j,k \in \N_0$.

(ii)$\Rightarrow$(iii):
From Theorem \ref{thm:Toeplitz+Hankel_bounded} we conclude that 
$A=T(\varphi)+H(\psi)$ can be written as $A=T(\hat{a})+H(\varrho)$ with $a\in M^p$, $\varrho\in c_0(\N)$ and $H(\varrho)$ being bounded. 

The uniqueness of $a$ and $\varrho$ follows from the uniqueness in Theorem \ref{thm:Toeplitz+Hankel_bounded},
and the statements about the norms follow from Theorem \ref{thm:Toeplitz+Hankel_norm}.
\end{proof}

We also have the following analogue for operators on $\ell^p(\Z)$.

\begin{prop} \label{prop:unique_decomposition-sequence}
For $1\le p\le \infty$, let $A$ be a bounded linear operator on $\ell^p(\Z)$ with matrix representation $A=(A_{j,k})_{j,k\in\Z}$ and
\begin{equation}\label{TH.displacement.infinite}
A_{j-1,k}+A_{j+1,k}-A_{j,k-1}-A_{j,k+1}=0 \quad (j,k\in\Z).
\end{equation}
Then $A=L(\hat{a})+L(\hat{b})J$ for some $a,b\in M^p$ and this decomposition is unique.
\end{prop}
\begin{proof}
Analogous to Theorem~\ref{main1}.
\end{proof}

\section{Hardy spaces}\label{Hardy spaces}

In this section we consider Toeplitz+Hankel operators on Hardy spaces $H^p$ ($1<p<\infty$) and prove analogous characterizations for their boundedness and compactness as in the previous section. Unfortunately, we are currently unable to deal with Toeplitz+Hankel operators on the Hardy space $H^1$ but conjecture that also in this case the boundedness (compactness) of $T(a)+H(b)$ implies that both $T(a)$ and $H(b)$ are bounded (compact). We will also consider multiplication operators with flip $M(a)+M(c)J$ on $L^p$. The corresponding results are valid for $1\le p<\infty$.

For $j\in\Z$, define $\chi_j \from \T \to \C$ by $\chi_j(z) = z^j$ and let $U^j = M(\chi_j)$.

\begin{lem} \label{lem:a_bounded}
Let $a \from \T \to \C$ be a measurable function and assume that there is a constant $C > 0$ such that
\[\norm{a(1-\chi_{2l})}_{\infty} \leq C\]
for all $l \in \N$. Then $a \in L^{\infty}$ and $\|a\|_{\infty}\le C/2$.
\end{lem}

\begin{proof}
Let $z \in \T$ such that $z^{2l} \neq 1$ for every $l \in \Z\setminus\{0\}$. Then $\{z^{2l}:l\in\Z\}$ is dense in $\T$, and the assumption implies
\[\abs{a(z)} \leq \inf\limits_{l \in \Z} \frac{C}{\abs{1-z^{2l}}} = \frac{C}{\sup\limits_{l \in \Z} \abs{1-z^{2l}}} =  \frac{C}{2}.\]
Since almost all $z\in\T$ satisfy the afore-mentioned condition, the assertion follows.
\end{proof}

For measurable functions $f,g \from \T \to \C$ such that $f \cdot g \in L^1$ we define
\[
\sp{f}{g} := \frac{1}{2\pi}\int_{\T} f(e^{i\theta})\overline{g(e^{i\theta})} \, \mathrm{d}\theta.
\]
For $a \in L^{\infty}$ denote by $M(a)$ the operator of multiplication by $a$. We recall the following useful characterization of multiplication operators on $L^p$. While it is stated for $1<p<\infty$ in \cite[Proposition~2.2]{BoeSi2006}, its proof remains valid also for $p=1$.

\begin{prop}\label{multiplication characterization}
For $1\le p<\infty$, let $A$ be a bounded linear operator on $L^p$, and suppose that there are complex numbers $\varphi_j$ such that
$$	
	\langle A\chi_k, \chi_j\rangle = \varphi_{j-k}
$$
for $j,k\in \Z$, that is, $A$ is of Laurent structure. Then there is an $a\in L^\infty$ such that $A = M(a)$, $a_j = \varphi_j$ for all $j\in \Z$, and
$$
	\|M(a)\| = \|a\|_\infty.
$$
\end{prop}

\begin{prop} \label{prop:unique_decomposition}
For $1\le p<\infty$, let $A$ be a bounded linear operator on $L^p$ and assume that
\begin{equation} \label{eq:prop_unique_decomposition}
\sp{A\chi_{k-1}}{\chi_j} - \sp{A\chi_k}{\chi_{j-1}} - \sp{A\chi_k}{\chi_{j+1}} + \sp{A\chi_{k+1}}{\chi_j} = 0
\end{equation}
for all $j,k \in \Z$. Then $A = M(a) + M(c)J$ for some $a,c \in L^{\infty}$ and this decomposition is unique.
Furthermore,
$$
\max\{\norm{a}_{\infty},\norm{c}_\infty\}\le \norm{A}\le \norm{a}_\infty+\norm{c}_\infty.
$$
\end{prop}

\begin{proof}
Observe
\begin{align*}
\sp{(A-UAU)\chi_{k-1}}{\chi_j} &= \sp{A\chi_{k-1}}{\chi_j} - \sp{A\chi_k}{\chi_{j-1}}\\
&= \sp{A\chi_k}{\chi_{j+1}} - \sp{A\chi_{k+1}}{\chi_j}\\
&= \sp{(A-UAU)\chi_k}{\chi_{j+1}}
\end{align*}
for all $j,k \in \Z$, i.e., $A-UAU$ has Laurent structure and is bounded, hence has to be equal to $M(\tilde{a}_1)$ for a bounded symbol $\tilde{a}_1$ with $\norm{\tilde{a}_1}_{\infty} \leq 2\norm{A}$ by Proposition~\ref{multiplication characterization}. More generally, for $l\in\N$, consider
\begin{align*}
A-U^lAU^l &=
\sum_{j=0}^{l-1} U^j(A-UAU)U^j
=\sum_{j=0}^{l-1} M(\chi_j)M(\tilde{a}_1)M(\chi_j)=M(\tilde{a}_l)
\end{align*}
with $\tilde{a}_l(z):=(1+z^2+\dots +z^{2(l-1)})\tilde{a}_1(z)=\frac{1-z^{2l}}{1-z^2}\tilde{a}_1(z)$,
which has to satisfy the estimate $\norm{\tilde{a}_l}_{\infty} \leq 2\norm{A}$ again by Proposition~\ref{multiplication characterization}. Defining 
$$
a(z):=\frac{\tilde{a}_1(z)}{1-z^2}=\frac{\tilde{a}_l(z)}{1-z^{2l}}
$$
for almost every $z\in\T$, it follows that
\[\norm{a(1-\chi_{2l})}_{\infty} = \norm{\tilde{a}_l}_{\infty} \leq 2\norm{A}.\]
for every $l\in \N$. By Lemma \ref{lem:a_bounded}, $a$ is essentially bounded and $\norm{a}_\infty\le \norm{A}$.

Moreover,
\begin{equation} \label{eq:prop_unique_decomposition_2}
	M(a) - UM(a)U = M(a(1-\chi_2)) = M(\tilde{a}_1) = A - UAU.
\end{equation}

Now consider $B := (A - M(a))J$, which is bounded on $L^p$. Then
\[
	UBU^{-1} = U(A-M(a))JU^{-1} = U(A-M(a))UJ = (A-M(a))J = B
\]
by \eqref{eq:prop_unique_decomposition_2}, i.e., $B$ has Laurent structure. This means that there is a $c \in L^{\infty}$ such that $B = M(c)$. Hence, $A = M(a) + M(c)J$.

To show uniqueness assume that there are $\tilde{a},\tilde{c} \in L^{\infty}$ such that
\[
	M(\tilde{a}) + M(\tilde{c})J = A = M(a) + M(c)J.
\]
Then $M(a-\tilde{a}) = M(\tilde{c}-c)J$ and hence
\begin{align*}
	\sp{a-\tilde{a}}{\chi_{k}} &= \sp{M(a-\tilde{a})\chi_0}{\chi_{k}} = \sp{M(\tilde{c}-c)J\chi_0}{\chi_{k}}\\
	&= \sp{M(\tilde{c}-c)\chi_{-1}}{\chi_{k}}
	= \sp{M(\tilde{c}-c)\chi_0}{\chi_{k+1}}\\
	&= \sp{M(\tilde{c}-c)J\chi_{-1}}{\chi_{k+1}}
	= \sp{M(a-\tilde{a})\chi_{-1}}{\chi_{k+1}}\\
	&= \sp{M(a-\tilde{a})\chi_0}{\chi_{k+2}}
	= \sp{a-\tilde{a}}{\chi_{k+2}}
\end{align*}
for all $k \in \Z$. It follows $a = \tilde{a}$ and $c = \tilde{c}$ almost everywhere by the Riemann--Lebesgue lemma.

Regarding the norm estimates, we have already shown $\norm{a}_\infty\le \norm{A}$. By considering the operator $AJ=M(a)J+M(c)$ instead of 
$A=M(a)+M(c)J$ and noting that it also satisfies \eqref{eq:prop_unique_decomposition}, we can repeat the arguments at the beginning of the proof. Correspondingly, we conclude that $AJ=M(\tilde{a})+M(\tilde{c})J$ with certain $\tilde{a},\tilde{c}\in L^\infty$ and $\norm{\tilde{a}}_\infty\le \norm{AJ}=\norm{A}$. The uniqueness of the representation entails that 
$c=\tilde{a}$, $a=\tilde{c}$, and thus $\norm{c}_\infty\le \norm{A}$. The upper norm estimates are obvious.
\end{proof}

\begin{rem}
It is straightforward to verify that, conversely, an operator $A=M(a)+M(c)J$ with $a,c\in L^\infty$ satisfies condition   \eqref{eq:prop_unique_decomposition}.  The norm of such an operator can even be evaluated exactly.
This uses the decomposition of $L^p$ into the direct sum $L^p(\T_+)\dotplus J L^p(\T_+)$, where $\T_+=\{z\in\T\,:\, \mathrm{Im}(z)>0\}$.
The operator can then be identified with a block multiplication operator with  $2\times 2$ matrix-valued symbol.
Indeed,
$$
\norm{A}_{B(L^p)}=  \esssup_{z\in\T_+} \norm{\left(\begin{array}{cc} a(z) & c(z) \\ c(\bar{z}) & a(\bar{z}) \end{array}\right)}_{p}
$$
where $\norm{\cdot}_p$ stands for the  $p$-norm of a matrix. We leave the detailed verification 
 to the reader. 
 \hfill $\Box$
\end{rem}

In the following we will consider sesquilinear forms $(\cdot,\cdot) \from \Dc_+ \times \Dc_+ \to \C$ with domain 
\[
	\Dc_+ := \spann\set{\chi_j : j \in \N_0}.
\]
Every densely defined operator $A \from H^p \to H^p$ with domain $\Dc(A) \supset \Dc_+$ induces a form via $(\cdot,\cdot)_A := \sp{A\cdot}{\cdot}$. To keep the notation slick, we will identify $A$ with $(\cdot,\cdot)_A$ if $(\cdot,\cdot)_A$ is the form induced by $A$. Conversely, if $(\cdot,\cdot)_A$ is a form, we will write $\sp{A\cdot}{\cdot} := (\cdot,\cdot)_A$ even if 
the form is not induced by an operator $A$.
Addition of forms and right multiplication with operators that leave $\Dc_+$ invariant are defined in the obvious way.

A form $T$ is called a {\em Toeplitz form} or {\em form of Toeplitz type} if $\sp{T\chi_k}{\chi_j} = \sp{T\chi_{k+1}}{\chi_{j+1}}$ for all $j,k \in \N_0$. A form $H$ is called a {\em Hankel form} or {\em form of Hankel type} if $\sp{H\chi_{k+1}}{\chi_j} = \sp{H\chi_k}{\chi_{j+1}}$ for all $j,k \in \N_0$. 
If there exists a function  $a\in L^1$ such that for $j,k\in\N_0$
\[
\sp{T\chi_k}{\chi_j} =\sp{a}{\chi_{j-k}}\quad\mbox{ or }\quad
\sp{H\chi_k}{\chi_j}=\sp{a}{\chi_{1+j+k}},
\]
respectively, then we say that the Toeplitz or Hankel form is induced by $a\in L^1$ and
we will write $T=T(a)$ and $H=H(a)$.
If a form is both a Toeplitz and a Hankel form we call it a {\em checkerboard form}. Clearly, the only bounded checkerboard form is $0$.

\begin{thm} \label{thm:Toeplitz+Hankel_bounded_Hardy}
For $1<p<\infty$, let $T$ and $H$ be forms of Toeplitz and Hankel type, respectively, and assume that $A = T + H$ is a bounded operator on $H^p$. 
Then \[T = T_1 + C\quad\mbox{and}\quad H=H_1-C,\]
where $T_1$ is a bounded Toeplitz operator, $H_1$ is a bounded Hankel operator, and $C$ is a checkerboard form.
Both decompositions are unique.

Furthermore, $T_1=T(a)$ and $H_1=H(b)$ for some $a,b\in L^\infty$, where the functions $a$ and $P(\chi_{-1}b)$ are uniquely determined.
\end{thm}

\begin{proof}
Let $P \from L^p \to H^p$ be the Riesz projection and $P^* \from H^p\to L^p$ be the inclusion. 
As the operators $(U^{-n}P^*APU^n)_{n\in\N}$ have a uniformly bounded norm, a standard argument shows that one can choose an increasing sequence $(n_l)_{l \in \N}$ such that the sequence $(\langle U^{-n_l}P^*APU^{n_l}\chi_k, \chi_j\rangle)_{l\in \N}$ converges for every  $j,k\in\Z$. As a consequence, the sequence $(U^{-n_l}P^*APU^{n_l})_{l \in \N}$ converges weakly to some bounded linear operator $B :L^p \to L^p$. 

For $j,k \in \N_0$ let $t_{j-k} := \sp{T\chi_k}{\chi_j}$ and $h_{j+k+1} := \sp{H\chi_k}{\chi_j}$. Then
\[
\lim\limits_{l \to \infty} \sp{U^{-n_l}P^*TPU^{n_l}\chi_k}{\chi_j} = \lim\limits_{l \to \infty} \sp{T\chi_{k+n_l}}{\chi_{j+n_l}} = t_{j-k}
\]
for all $j,k \in \Z$. It follows that the limits $\lim\limits_{l \to \infty} \sp{U^{-n_l}P^*HPU^{n_l}\chi_k}{\chi_j}$ also exist and
\[
\lim\limits_{l \to \infty} \sp{U^{-n_l} P^*HPU^{n_l}\chi_k}{\chi_j} =  \lim\limits_{l \to \infty} \sp{H\chi_{k+n_l}}{\chi_{j+n_l}} = \lim\limits_{l \to \infty} h_{j+k+2n_l+1}
\! =: \!\tilde{h}_{j+k}\]
for all $j,k \in \Z$. Combining the previous two equations, we get 
\[
\sp{B\chi_k}{\chi_j}=t_{j-k}+\tilde{h}_{j+k},
\]
and therefore $B$ satisfies \eqref{eq:prop_unique_decomposition}. Hence, by Proposition \ref{prop:unique_decomposition}, 
$B = M(a)+M(c)J$ for some $a,c \in L^{\infty}$. In particular, for $j,k \in \N_0$, 
\[
t_{j-k} - \sp{T(a)\chi_k}{\chi_j} 
= t_{j-k} - \sp{M(a)\chi_k}{\chi_j} = \sp{M(c)J\chi_k}{\chi_j} - \tilde{h}_{j+k}. 
\]
The left-hand side is invariant under $j \mapsto j+1$, $k \mapsto k+1$. The right-hand side is invariant under $j \mapsto j-1$, $k \mapsto k+1$. It follows that $C:=T-T(a)$ is of checkerboard type and $T_1 := T(a)$ is bounded. 
Consequently, $H_1:=H+C=H+T-T_1   =A-T_1$ is a Hankel form and bounded on $H^p$. This yields the desired decompositions.
 The uniqueness follows from the fact that a checkerboard form is necessarily unbounded or zero.

Finally, the characterization of bounded Hankel operators  \cite[Theorem 2.11]{BoeSi2006} implies that
$H_1=H(b)$ with some (non-unique) $b\in L^\infty$, and we know already $T_1=T(a)$. Clearly, the functions $a$ and $P(\chi_{-1}b)$ are uniquely determined as their Fourier coefficients are given by $T_1$ and $H_1$.
\end{proof}

The reasoning in the previous proof breaks down in the case $p=1$. Indeed, as $P$ is not bounded on $H^1$, we cannot conclude $a \in L^{\infty}$ as easily in the first part and even if we had $a \in L^{\infty}$, the corresponding Toeplitz operator $T(a)$ needs not be bounded on $H^1$. As mentioned at the beginning of this section, we conjecture that for $p=1$ the boundedness of $T(a)$ can be shown by different means.

\begin{thm} \label{thm:Toeplitz+Hankel_compact_Hardy_1}
For $1<p<\infty$, consider the  bounded linear operator  $A=T(a)+H(b)$ on $H^p$ with $a,b\in L^\infty$. Then, for any compact operator $K$,
\[
\max\{\norm{a}_\infty,\frac{1}{1+c_p}\norm{H(b)+K}\}   \le\norm{A+K} \le c_p\norm{a}_\infty+\|H(b)+K\|
\]
where $c_p=\norm{P}_{B(L^p)}$.
\end{thm}

We recall that the constant $c_p$ is equal to $\frac{1}{\sin(\frac{\pi}{p})}$ (see \eqref{e:c_p}).  

\begin{proof}
In the proof of Theorem \ref{thm:Toeplitz+Hankel_bounded_Hardy}  it was shown that $B=M(a)+M(c)J$ is a bounded linear operator on $L^p$, which is the weak limit of
the sequence of operators $(U^{-n_l}P^*APU^{n_l})_{l \in \N}$.
As $(U^n)_{n \in \N}$ converges weakly to $0$ and $K$ is compact, $(KPU^n)_{n \in \N}$ converges strongly to $0$. It follows that $(U^{-n}P^* KPU^n)_{n \in \N}$ converges strongly to $0$ as well. Therefore, $B$ is the weak limit of the sequence 
 $(U^{-n_l}P^*(A+K)PU^{n_l})_{l \in \N}$, and this implies that for each Laurent polynomial $f$,
\begin{align*}
\norm{Bf}_{L^p} &\le \liminf_{l\to\infty} \norm{U^{-n_l}P^*(A+K)PU^{n_l}f}_{L^p}\\
&\le  \norm{A+K}_{B(H^p)} \liminf_{l\to\infty} \norm{PU^{n_l}f}_{H^p} =\norm{A+K}_{B(H^p)}\norm{f}_{L^p}.
 \end{align*}
We deduce that $\norm{B}\le \norm{A+K}$, and combining this with the norm estimate $\norm{a}_{\infty} \le \norm{B}$ 
established in Proposition \ref{prop:unique_decomposition} we obtain
 $\norm{a}_{\infty} \le  \norm{A+K}$, which is part of the lower estimate. For the upper estimate notice that $\norm{T(a)}\le c_p\norm{a}_{\infty}$.
Furthermore,
\begin{align*}
\norm{H(b)+K}\le \norm{A+K}+\norm{T(a)}\le (1+c_p)\norm{A+K},
\end{align*}
which completes the lower estimate.
\end{proof}

\begin{cor}\label{cor:Toeplitz+Hankel_compact_Hardy_2}
For $1<p<\infty$, let $T$ and $H$ be forms of Toeplitz and Hankel type, respectively, and assume that $A = T + H$ is a compact operator on $H^p$. Then $T=C$ is a checkerboard form and $H = H(b)-C$, where $H(b)$ is a compact Hankel with $b\in C(\T)$.
\end{cor}

\begin{proof}
By Theorem \ref{thm:Toeplitz+Hankel_bounded_Hardy} we can write $T=T(a)+C$ and $H=H(b)-C$
with $a,b\in L^\infty$ and $C$ a checkerboard form.  Thus $A=T(a)+H(b)$, and Theorem \ref{thm:Toeplitz+Hankel_compact_Hardy_1} with $K=-A$ implies that $a=0$ and that $H(b)=-K$ is compact. The characterization of compact Hankel operators \cite[Theorem 2.54]{BoeSi2006} completes the proof.
\end{proof}

\begin{cor} \label{cor:Toeplitz+Hankel_bounded_cmp_Hardy}
For $1<p<\infty$, let $A = T(a) + H$ where $a \in L^1$ and $H$ is a Hankel form. 
\begin{itemize}
\item[(i)]
If $A$ is bounded on $H^p$, then $a \in L^{\infty}$ and $H = H(b)$ for some $b \in L^{\infty}$.
\item[(ii)]
If $A$ is compact on $H^p$, then $a =0$ and $H = H(b)$ for some $b \in C(\T)$.
\end{itemize}
\end{cor}

\begin{proof}
In view of Theorem \ref{thm:Toeplitz+Hankel_bounded_Hardy} and Corollary \ref{cor:Toeplitz+Hankel_compact_Hardy_2},
it suffices to show that if $T(c)$ with $c\in L^1$  is a checkerboard form, then $c = 0$. But this follows directly from the Riemann--Lebesgue lemma.
\end{proof}

Finally, let us summarize the previous results by establishing equivalent characterizations of bounded Toeplitz + Hankel forms.

\begin{thm}\label{thm:main2}
Let $1< p<\infty$. Then the following statements are equivalent:
\begin{itemize}
\item[{\rm (i)}]
$A$ is a bounded linear operator on $H^p$ satisfying the relations
\begin{equation*}
\sp{A\chi_{j-1}}{\chi_k} - \sp{A\chi_j}{\chi_{k-1}} - \sp{A\chi_j}{\chi_{k+1}} + \sp{A\chi_{j+1}}{\chi_k} = 0
\end{equation*}
for all $j,k \in \N$.
\item[{\rm (ii)}]
$A$ is a bounded linear operator on $H^p$ which is a sum of a Toeplitz and a Hankel form.
\item[{\rm (iii)}]
$A=T(a)+H(b)$ for some $a,b\in L^\infty$.
\end{itemize}
Furthermore, the functions $a$ and $P(\chi_{-1}b)$ are uniquely determined, and 
\begin{align*}
\norm{T(a)+H(b)}  & \simeq
\max\{ \norm{a}_{\infty},\norm{H(b)} \},
\\[1ex]
\norm{T(a)+H(b)}_{\mathrm{ess}}  &\simeq
\max\{ \norm{a}_{\infty},\norm{H(b)}_{\mathrm{ess}} \},
\end{align*}
where ``$\simeq$'' stands for equivalence of norms.
\end{thm}
\begin{proof}
Similar to the proof of Theorem~\ref{main1}.
\end{proof}

Elaborating on the last part of the previous theorem, we remark that the norm and the essential norm of a Hankel operator $H(b)$ on $H^p$ is equivalent to
\[
\dist_{L^\infty}(b,\overline{H^\infty}) \quad\mbox{ and }\quad
\dist_{L^\infty}(b,C(\T)+\overline{H^\infty}),
\]
respectively. For details see \cite[Theorems~2.11 and 2.54]{BoeSi2006} and the comments and references given there.

\subsection*{Acknowledgements}
This project has received funding from the European Union's Horizon 2020 research and innovation programme under the Marie Sklodowska-Curie grant agreement No 844451. Ehrhardt was supported in part by the Simons Foundation Collaboration Grant \# 525111. Virtanen was supported in part by EPSRC grant EP/T008636/1. The authors also thank the American Institute of Mathematics and the SQuaRE program for their support.

\end{document}